\documentclass{amsart}[12pt]

\usepackage{amsmath}
\usepackage{amsfonts}
\usepackage{amssymb}

\newtheorem{theorem}{Theorem}
\newtheorem{lemma}{Lemma}
\newtheorem{proposition}{Proposition}
\newtheorem{corollary}{Corollary}

\theoremstyle{definition}

\newtheorem{remark}{Remark}

\begin{document}
\title[Invariant subspaces of shift plus Volterra operator]%
{The invariant subspaces of the shift plus integer multiple of Volterra operator on Hardy spaces}

\author{Qingze Lin}

\address{School of Applied Mathematics, Guangdong University of Technology, Guangzhou, Guangdong, 510520, P.~R.~China}\email{gdlqz@e.gzhu.edu.cn}

\begin{abstract}
\v{C}u\v{c}kovi\'{c} and Paudyal recently characterized the lattice of invariant subspaces of the shift plus a complex Volterra operator on the Hilbert space $H^2$ on the unit disk. Motivated by the idea of Ong, in this paper, we give a complete characterization of the lattice of invariant subspaces of the shift operator plus a positive integer multiple of the Volterra operator on Hardy spaces $H^p$, which essentially extends their works to the more general cases when $1\leq p<\infty$.
\end{abstract}

\maketitle

\section{\bf Introduction}
It is still an open problem whether every bounded linear operator on an infinite dimensional separable complex Hilbert space has a nontrivial invariant subspace (see \cite{CP,BSY} for comprehensive introductions and the recent developments). During the 1930s, John von Neumann proved that every completely continuous (i.e., compact) operator on a Hilbert space has a nontrivial invariant subspace, but he did not publish it while the same result was published by Aronszajn and Smith in \cite{AS}. The result was extended to polynomially compact operators by Bernstein and Robinson \cite{BR} by utilizing the non-standard analysis, which was translated into standard analysis by Halmos \cite{PRH} in the same issue of the same journal.

Although lots of famous mathematicians had made great efforts on this open problem for many years, the operators, of which the invariant subspaces can be characterized successfully, are still very few (see \cite{ENFLO,READ1,READ2}). However, the first successful characterization of the invariant subspaces was obtained by Beurling's paper \cite{BEURLING}, in which he characterized all the invariant subspaces of the shift operator $(M_zf)(z)=zf(z)$
acting on the Hilbert space on the unit disk of the complex plane. After his work, the lattice of invariant subspaces of the shift operator on the Hardy space was also obtained (see \cite{DUREN} for example). It is of interest that Cowen and Wahl \cite{CW} and Matache \cite{VM} recently studied relationships between shift-invariant subspaces of Hilbert space and invariant subspaces of composition operators. Although the Beurling's theorem in Hardy spaces is not true in the Bergman spaces, Aleman, Richter and Sundberg \cite{ARS} proved that every cyclic invariant subspace of shift operator on the Bergman space $L_a^p\ ,0<p<\infty,$ is generated by its extremal function.

Let $\mathbb{D}$ be the unit disk of the complex plane and $H(\mathbb D)$ be the space consisting of all analytic functions on $\mathbb{D}$. For $0<p<\infty$, the Hardy space $H^{p}$ on the unit disk $\mathbb{D}$ consists of all functions $f\in H(\mathbb D)$ satisfying
$$\|f\|_{H^{p}}=\Big(\sup_{0<r<1}\frac{1}{2\pi}\int_{0}^{2\pi}|f(re^{i\theta})|^{p}d\theta\Big)^{1/p}<\infty\ .$$
For $p=\infty$, the space $H^{\infty}$ is defined as
$$H^{\infty}=\{f\in H(\mathbb D):\ \|f\|_{\infty}=\sup_{z\in \mathbb D}\{|f(z)|\}<\infty\}\,.$$
Then, for any function $g\in H(\mathbb D)$, the Volterra type operator $T_g$ is defined by
$$(T_gf)(z)=\int_0^z f(\omega)g'(\omega)d\omega\qquad \text{ for any }f\in H(\mathbb D)\,.$$

In 2008, Aleman and Korenblum \cite{AK} gave a complete description of the invariant subspaces of $H^p$ under the Volterra type operator, and as an application, they gave the characterizations of the invariant subspaces of the Volterra type operator on a large class of Banach spaces of analytic functions in the unit disk containing the Bergman spaces, Dirichlet spaces and Besov spaces and so forth.

Sarason \cite{SARASON1,SARASON2} used the Beurling's ideas to give the lattice of invariant subspaces of the space $L^2(0,1)$ under the real version of the shift operator plus a Volterra operator. Using the idea of Sarason, Montes-Rodr\'{\i}guez, Ponce-Escudero and Shkarin \cite{MRPE} provided a precise description of the lattice of invariant subspaces of composition operators on the Hilbert space, whose inducing symbol is a parabolic non-automorphism. Following Sarason's work as well, \v{C}u\v{c}kovi\'{c} and Paudyal \cite{ZCBP} characterized the lattice of invariant subspaces of shift operator plus complex Volterra operator on the Hilbert space $H^2$. Their idea is to transform the problems of invariant subspaces of shift operator plus complex Volterra operator on $H^2$ into the problems of invariant subspaces of $M_z$ on the space $S^2$ consisting of analytic functions with derivative in $H^2$. Recently, their results were partly extended by Lin, Liu and Wu \cite{LIN}.

In the present work, motivated by the works of Ong \cite{ONG} who studied the extensions of Sarason's results and by the works of \v{C}u\v{c}kovi\'{c} and Paudyal \cite{ZCBP1} who study the complex version of their results, we extend the results of \v{C}u\v{c}kovi\'{c} and Paudyal \cite{ZCBP1} to the more general cases when $1\leq p<\infty$\,, and in the meantime, simplify their proofs.

For $1\leq p<\infty$, denote $S_n^p$ the space of analytic functions with $n$th-derivative in $H^p$, that is,
$$S^p_n=\{f\in H(\mathbb{D}):\ f^{(n)}\in H^p\}\,,\qquad \text{ where } n\in \mathbb{N}\,,$$
endowed with the norm defined, recursively, by
$$\|f\|_{S_n^p}=|f(0)|+\|f'\|_{S_{n-1}^p} \qquad \text{ where } S_0^p\equiv H^p\,.$$
 Then we define the subspace of $S_{n}^p$ by:
$$S_{n,0}^p=\{f\in S_{n}^p:\ f^{(m)}(0)=0,\ 0\leq m\leq n-1\}\,.$$

In Section~2, we give some definitions of several concepts and present the main results of this paper. In Section~3, we characterize the properties of the spaces $S_n^p$ and the lattices of invariant subspaces of $M_z$ on $S_n^p$ and $S_{n,0}^p$\,. Then in Section~4, we complete the proof of the main results.

\section{\bf Main results}
Given any $n\in \mathbb{N}$, we introduce an operator, which is the sum of the shift plus a positive integer multiple of the Volterra operator:
$$Tf=M_zf+nT_{z}f\,.$$
This operator is the operator recently defined by \v{C}u\v{c}kovi\'{c} and Paudyal in \cite{ZCBP,ZCBP1}.

Now if $n\in\mathbb{N}$, consider the closed subsets $K_0,K_1,\ldots,K_{n-1}$ of $\partial\mathbb{D}$ and an inner function $G$ with the following four properties:\\
\phantom{(1)}(1) $K_0\supset K_1\supset\cdots\supset K_{n-1}$;\\
\phantom{(1)}(2) $K_0\backslash K_{n-1}$ is isolated;\\
\phantom{(1)}(3) the zeros of $G$ in $\mathbb{D}$ cluster only in $K_{n-1}$;\\
\phantom{(1)}(4) the singular measure associated with $G$ is supported by $K_{n-1}$\,.

Then, if $A$ is a \mbox{Banach} algebra, we denote by $J_{A}(G;K_0,K_1,\ldots,K_{n-1})$ the set of $f\in A$ (see \cite{BK,SF} for the details), satisfying the above four properties, such that\\
\phantom{(1)}(1) $f^{(j)}(z)=0\,, \text{ whenever }z\in K_j$, $0\leq j\leq n-1$;\\
\phantom{(1)}(2) the inner factor of $f$ is divisible by $G$.

Now we can state our main results as follows:
\begin{proposition}\label{pro1}
Let $n\in\mathbb{N}$ and $1\leq p<\infty$. Then the subspace $\mathcal{J}$ is invariant for $M_z$ on $S_{n}^p$ if and only if there exist $G,K_0,\ldots,K_{n-1}$ such that
$$\mathcal{J}=J_{S_{n}^p}(G;K_0,K_1,\ldots,K_{n-1})\,.$$
\end{proposition}

\begin{proposition}\label{pro2}
Let $n\in\mathbb{N}$ and $1\leq p<\infty$. Then the subspace $\mathcal{J}$ is invariant for $M_z$ on $S_{n,0}^p$ if and only if there exist $G,K_0,\ldots,K_{n-1}$ such that
$$\mathcal{J}=J_{S_{n,0}^p}(G;K_0,K_1,\ldots,K_{n-1})\,.$$
\end{proposition}

\begin{theorem}\label{th1}
Let $n\in\mathbb{N}$ and $1\leq p<\infty$. Then the subspace $\mathcal{J}$ is invariant for $T$ on $H^p$ if and only if there exist $G,K_0,\ldots,K_{n-1}$ such that
$$\mathcal{J}=\{f^{(n)}:f\in J_{S_{n,0}^p}(G;K_0,\ldots,K_{n-1})\}\,.$$
\end{theorem}

\section{\bf The lattice of invariant subspaces of $M_z$ on $S_n^p$ and $S_{n,0}^p$}
In this section, we are to prove Proposition~\ref{pro1} and Proposition~\ref{pro2}.
First, we need to characterize some properties of the spaces $S_n^p$ which will be essential for the subsequent discussions. In order to prove the boundedness of elements of the spaces $S_n^p$, the following important inequality will be very useful (see \cite[page~48]{DUREN} for its proof).
\begin{lemma}[Hardy's inequality]\label{le1}
If $f\in H^1$ and $f=\sum_{n=0}^\infty a_nz^n$, $z\in\mathbb{D}$, then
$$\sum_{n=0}^\infty\frac{|a_n|}{n+1}\leq\pi\|f\|_{H^1}\,.$$
\end{lemma}

\begin{proposition}\label{pro3}
Let $n\in\mathbb{N}$ and $1\leq p<\infty$. Then\\
\phantom{(1)}\textup{(1)} $\|f\|_{\infty}\leq\pi\|f\|_{S_1^p}\leq\pi^2\|f\|_{S_2^p}\leq\ldots\leq\pi^n\|f\|_{S_n^p}\,;$\\
\phantom{(1)}\textup{(2)} $S_n^p\subset S_{n-1}^p\subset\ldots\subset S_1^p\subset H^\infty\,.$
\end{proposition}
\begin{proof}
We just need to prove (1), since (2) follows from (1) immediately. By Lemma~\ref{le1}, if $f\in S_1^p$, then
$$|f(z)|\leq\pi\|f'\|_{H^1}+|f(0)|\leq\pi\|f\|_{S_1^p}\,,\quad z\in\mathbb{D}\,.$$
Hence, $\|f\|_{\infty}\leq\pi\|f\|_{S_1^p}$\,. Now if $f\in S_n^p$, then
\begin{equation}\begin{split}\nonumber
\pi\|f\|_{S_n^p}&=\pi\left(|f(0)|+|f'(0)|+\cdots+|f^{(n-2)}(0)|+\|f^{(n-1)}\|_{S_1^p}\right)\\
&\geq |f(0)|+|f'(0)|+\cdots+|f^{(n-2)}(0)|+\|f^{(n-1)}\|_{\infty}\\
&\geq |f(0)|+|f'(0)|+\cdots+|f^{(n-2)}(0)|+\|f^{(n-1)}\|_{H^p}\\
&=\|f\|_{S_{n-1}^p}\,.
\end{split}\end{equation}
Thus, (1) follows by recursion immediately.
\end{proof}
By Proposition~\ref{pro3}, we see that for any $n\in\mathbb{N}$ and $1\leq p<\infty$, the inclusion operator from $S_n^p$ into $H^\infty$ is bounded. Thus, we obtain the following two equivalent norms for $S_n^p$\,.
\begin{corollary}\label{cor1}
Let $n\in\mathbb{N}$ and $1\leq p<\infty$. Then there exist two other equivalent norms for $S_n^p$ defined, respectively, by:
$$\|f\|=\|f\|_{H^p}+\|f'\|_{H^p}+\cdots+\|f^{(n-1)}\|_{H^p}+\|f^{(n)}\|_{H^p}\,,$$
and
$$\|f\|=\|f\|_{\infty}+\|f'\|_{\infty}+\cdots+\|f^{(n-1)}\|_{\infty}+\|f^{(n)}\|_{H^p}\,.$$
\end{corollary}

And from Corollary~\ref{cor1}, it follows that when $n\in\mathbb{N}$ and $1\leq p<\infty$, $S_n^p$ is a \mbox{Banach} space. In what follows, we are to prove furthermore that $S_n^p$ is also a \mbox{Banach} algebra with the pointwise multiplications.

\begin{proposition}\label{pro4}
Let $n\in\mathbb{N}$ and $1\leq p<\infty$. Then\\
\phantom{(1)}\textup{(1)} $S_n^p$ is a \mbox{Banach} algebra with the pointwise multiplications;\\
\phantom{(1)}\textup{(2)} polynomials are dense in $S_n^p$\,.
\end{proposition}
\begin{proof}
(1) For any $f,g\in S_n^p$, by Proposition~\ref{pro3}, we have
\begin{equation}\begin{split}\nonumber
\|fg\|_{S_n^p}&=|f(0)g(0)|+|(fg)'(0)|+\cdots+|(fg)^{(n-1)}(0)|+\|(fg)^{(n)}\|_{H^p}\\
&\leq |f(0)g(0)|+\left(|f'(0)g(0)|+|f(0)g'(0)|\right)+\cdots\\
&\phantom{\leq}+\sum_{k=0}^{n-1}\binom{n-1}{k}\left|f^{(k)}(0)g^{(n-1-k)}(0)\right|+\left\|\sum_{k=0}^{n}\binom{n}{k}f^{(k)}g^{(n-k)}\right\|_{H^p}\\
&\leq \|f\|_{S_n^p}\|g\|_{S_n^p}+\left(\|f'\|_{S_{n-1}^p}\|g\|_{S_n^p}+\|f\|_{S_{n}^p}\|g'\|_{S_{n-1}^p}\right)+\cdots\\
&\phantom{\leq}+\sum_{k=0}^{n-1}\binom{n-1}{k}\left(\|f^{(k)}\|_{S_{n-k}^p}\|g^{(n-1-k)}\|_{S_{k+1}^p}\right)+\\
&\phantom{\leq}+\|f^{(n)}g\|_{H^p}+\|fg^{(n)}\|_{H^p}+\sum_{k=1}^{n-1}\binom{n}{k}\left\|f^{(k)}g^{(n-k)}\right\|_{H^p}\\
\end{split}\end{equation}
\begin{equation}\begin{split}\nonumber
\phantom{\|fg\|_{S_n^p}}&\leq \|f\|_{S_n^p}\|g\|_{S_n^p}+\left(\|f\|_{S_{n}^p}\|g\|_{S_n^p}+\|f\|_{S_{n}^p}\|g\|_{S_{n}^p}\right)+\cdots\\
&\phantom{\leq}+\sum_{k=0}^{n-1}\binom{n-1}{k}\left(\|f\|_{S_{n}^p}\|g\|_{S_{n}^p}\right)+\\
&\phantom{\leq}+\|f^{(n)}\|_{H^p}\|g\|_{\infty}+\|f\|_{\infty}\|g^{(n)}\|_{H^p}+\sum_{k=1}^{n-1}\binom{n}{k}\|f^{(k)}\|_{\infty}\|g^{(n-k)}\|_{\infty}\\
&\leq \sum_{k=0}^{n-1}2^k\|f\|_{S_n^p}\|g\|_{S_n^p}+\pi^n\|f\|_{S_n^p}\|g\|_{S_n^p}+\pi^n\|f\|_{S_n^p}\|g\|_{S_n^p}+\\
&\phantom{\leq}+\sum_{k=1}^{n-1}\binom{n}{k}\pi^{n-k}\|f^{(k)}\|_{S_{n-k}^p}\pi^k\|g^{(n-k)}\|_{S_{k}^p}\\
&\leq (2^n-1)\|f\|_{S_n^p}\|g\|_{S_n^p}+\binom{n}{0}\pi^n\|f\|_{S_n^p}\|g\|_{S_n^p}+\binom{n}{n}\pi^n\|f\|_{S_n^p}\|g\|_{S_n^p}+\\
&\phantom{\leq}+\sum_{k=1}^{n-1}\binom{n}{k}\pi^{n}\|f\|_{S_{n}^p}\|g\|_{S_{n}^p}\\
&\leq (2^n(1+\pi^n)-1)\|f\|_{S_n^p}\|g\|_{S_n^p}\,.
\end{split}\end{equation}
Thus, it can be easily verified that $\|f\|=(2^n(1+\pi^n)-1)\|f\|_{S_n^p}$ is the equivalent norm that makes $S_n^p$ into a Banach algebra.

(2) If $f\in S_n^p$, then $f^{(n)}\in H^p$ by definition. Since polynomials are dense in $H^p$, there exists a sequence of polynomials, $\{p_m(z)\}_{m=0}^\infty$\,, such that $\|p_m-f^{(n)}\|_{H^p}\rightarrow0$ as $m\rightarrow\infty$\,. Let
$$P_m(z)=f(0)+f'(0)z+\cdots+\frac{f^{(n-1)}(0)}{(n-1)!}z^{n-1}+\frac{1}{(n-1)!}\int_0^z(z-\zeta)^{n-1}p_m(\zeta)d\zeta\,.$$
Then $P_m(z)$ is also a polynomial. What's more,
\begin{equation}\begin{split}\nonumber
\|P_m-f\|_{S_n^p}&=|P_m(0)-f(0)|+|P'_m(0)-f'(0)|+\cdots\\
&\phantom{=}+|P^{(n-1)}_m(0)-f^{(n-1)}(0)|+\|P^{(n)}_m-f^{(n)}\|_{H^p}\\
&=\|p_m-f^{(n)}\|_{H^p}\,.
\end{split}\end{equation}
It follows that $\|P_m-f\|_{S_n^p}\rightarrow0$ as $m\rightarrow\infty$\,.
\end{proof}

From Proposition~\ref{pro4} and the definition of $S_{n,0}^p$\,, we obtain the following corollary directly.
\begin{corollary}\label{cor2}
Let $n\in\mathbb{N}$ and $1\leq p<\infty$. Then $S_{n,0}^p$, endowed with the same norm defined for $S_n^p$, is a \mbox{Banach} subalgebra of $S_n^p$\,, and $S_{n,0}^p=\overline{span}\{z^{n+j}: j\in \mathbb{N}\cup\{0\}\}$\,.
\end{corollary}

The following two lemmas characterize the relationships between the invariant subspaces of the shift and the closed ideals of Banach algebra $S_{n}^p$ and $S_{n,0}^p$\,.

\begin{lemma}\label{le2}
Let $n\in\mathbb{N}$, $1\leq p<\infty$. Then the invariant subspaces of $M_{z}$ on $S_{n}^p$ are precisely the closed ideals of $S_{n}^p$\,.
\end{lemma}
\begin{proof}
Let $\mathcal{J}$ be an invariant subspace of $M_z$ on $S_{n}^p$\,. Then it is easily verified that the set
$$\mathcal{M}_{\mathcal{J}}=\{h\in S_{n}^p: hx\in\mathcal{J}\text{ whenever }x\in \mathcal{J}\}$$
is a (closed) subspace of $S_{n}^p$\,. Since for any $x\in \mathcal{J}$, $z^mx=M^m_zx\in\mathcal{J}$ whenever $m\in \mathbb{N}\cup\{0\}$, it follows that $z^m\ (m\in \mathbb{N}\cup\{0\})$ are all in $\mathcal{M}_{\mathcal{J}}$\,. Now it follows from Proposition~\ref{pro4} that $\mathcal{M}_{\mathcal{J}}=S_{n}^p$\,. Thus, $\mathcal{J}$ is a closed ideal of $S_{n}^p$\,.

Conversely, since $z$ is an element of $S_{n}^p$, it is clear that every closed ideal of $S_{n}^p$ is an invariant subspace of $M_{z}$ on $S_{n}^p$\,.
\end{proof}

\begin{lemma}\label{le3}
Let $n\in\mathbb{N}$, $1\leq p<\infty$. Then the invariant subspaces of $M_{z}$ on $S_{n,0}^p$ are precisely the closed ideals of $S_{n}^p$ contained in $S_{n,0}^p$\,.
\end{lemma}
\begin{proof}
It is obvious that if $\mathcal{J}$ is a closed ideal of $S_{n}^p$ contained in $S_{n,0}^p$\,, then $\mathcal{J}$ is an invariant subspaces of $M_{z}$ on $S_{n,0}^p$\,.

Conversely, let $\mathcal{J}$ be an invariant subspace of $M_z$ on $S_{n,0}^p$\,. Then it is easily verified that the set
$$\mathcal{M}_{\mathcal{J}}=\{h\in S_{n,0}^p: hx\in\mathcal{J}\text{ whenever }x\in \mathcal{J}\}$$
is a (closed) subspace of $S_{n,0}^p$\,. Since for any $x\in \mathcal{J}$, $z^mx=M^m_zx\in\mathcal{J}$ whenever $m\in \mathbb{N}\cup\{0\}$, it follows that $z^m\ (m\in \mathbb{N}\cup\{0\})$ are all in $\mathcal{M}_{\mathcal{J}}$\,. Now it follows from Proposition~\ref{pro4} that $\mathcal{M}_{\mathcal{J}}=S_{n}^p$\,. Thus, $\mathcal{J}$ is a closed ideal of $S_{n}^p$ contained in $S_{n,0}^p$\,.
\end{proof}

The following lemma, deduced from \cite{SF}, characterizes all the nontrivial closed ideals of $S_{n}^p$\,. It should be noted that Korenblum \cite{BK} had obtaind the special case of this lemma when $p=2$\,.
\begin{lemma}\label{le4}
Let $n\in\mathbb{N}$, $1\leq p<\infty$, and suppose that $\int_{0}^{2\pi}\log\rho(e^{i\theta})d\theta>-\infty$, where $\rho(z)$ is the distance of $z$ from the union of $K_0$ and the zero set of $G$\,. Then $J_{S_{n}^p}(G;K_0,K_1,\ldots,K_{n-1})$ is a nontrivial closed ideal of $S_{n}^p$\,. Conversely, If $\mathcal{J}$ is a nontrivial closed ideal of $S_{n}^p$\,, then there exist $G,K_0,\ldots,K_{n-1}$ such that
$$\mathcal{J}=J_{S_{n}^p}(G;K_0,K_1,\ldots,K_{n-1})\,.$$
\end{lemma}
\begin{remark}\label{re1}
Let $n\in\mathbb{N}$, $1\leq p<\infty$\,. Then $J_{S_{n}^p}(G;K_0,K_1,\ldots,K_{n-1})$ is a zero ideal if and only if $\int_{0}^{2\pi}\log\rho(e^{i\theta})d\theta>-\infty$, where $\rho(z)$ is the distance of $z$ from the union of $K_0$ and the zero set of $G$.
\end{remark}

Now, we can give the proof of Proposition~\ref{pro1} and Proposition~\ref{pro2} in Section~2.
\begin{proof}[Proof of Proposition~\ref{pro1}]
By Lemma~\ref{le2}, Lemma~\ref{le4} and Remark~\ref{re1}, the proposition holds immediately.
\end{proof}

\begin{proof}[Proof of Proposition~\ref{pro2}]
According to Corollary~\ref{cor2}, $S_{n,0}^p=z^nS_{n}^p$\,. Then by Lemma~\ref{le4} and Remark~\ref{re1}, the closed ideals of $S_{n}^p$ contained in $S_{n,0}^p$ are precisely of the form $J_{S_{n,0}^p}(G;K_0,K_1,\ldots,K_{n-1})$\,. Thus by Lemma~\ref{le3}, the proposition follows.
\end{proof}

\section{\bf The proof of the main results}
In this section, we want to complete the proof of the main theorem, Theorem~\ref{th1}.

\begin{lemma}\label{le5}
Let $n\in\mathbb{N}$, $1\leq p<\infty$\,. Then\\
\textup{(1)} The operator $\frac{d^n}{dz^n}$ is an isometrical isomorphism from $S_{n,0}^p$ onto $H^p$\,, and its inverse is $V_n$ such that
$$(V_nf)(z)=\frac{1}{(n-1)!}\int_0^z(z-\zeta)^{n-1}f(\zeta)d\zeta\;;$$
\textup{(2)} The operator $M_z$ is similar to $T$ under the operator $\frac{d^n}{dz^n}$\,, that is,\\ $$M_z=\left(\frac{d^n}{dz^n}\right)^{-1}T\frac{d^n}{dz^n}\,.$$
\end{lemma}
\begin{proof}
(1) The fact that $\frac{d^n}{dz^n}$ is an isometrical isomorphism from $S_{n,0}^p$ onto $H^p$ follows from the definition of norm for $S_{n,0}^p$\,. The remainder follows from the fact that if $f\in S_{n,0}^p$\,, then  $$f(z)=\frac{1}{(n-1)!}\int_0^z(z-\zeta)^{n-1}f^{(n)}(\zeta)d\zeta\,,$$
and that if $f\in H^p$\, then
$$\frac{d^n}{dz^n}\left(\frac{1}{(n-1)!}\int_0^z(z-\zeta)^{n-1}f(\zeta)d\zeta\right)=f(z)\,.$$

(2) Now, if $f\in S_{n,0}^p$\,, then for any $z\in \mathbb{D}$,
\begin{equation}\begin{split}\nonumber
\left(\frac{d^n}{dz^n}M_zf\right)(z)&=\frac{d^n}{dz^n}(zf(z))=\binom{n}{0}zf^{(n)}(z)+\binom{n}{1}f^{(n-1)}(z)\\
&=(M_zf^{(n)})(z)+(n T_{z}f^{(n)})(z)\\
&=\left(T\frac{d^n}{dz^n}f\right)(z)\,.
\end{split}\end{equation}
\end{proof}

Finally, we can give the proof of the main theorem, Theorem~\ref{th1}.
\begin{proof}[Proof of Theorem~\ref{th1}]
By Proposition~\ref{pro2}, it suffices to prove that $\frac{d^n}{dz^n}\left(\mathcal{J}\right)$ is an invariant subspace of $T$ on $H^p$ if and only if $\mathcal J$ is an invariant subspace of $M_z$ on $S_{n,0}^p$\,.

By Lemma~\ref{le5}, the operator $\frac{d^n}{dz^n}$ is an isometrical isomorphism from $S_{n,0}^p$ onto $H^p$\,, hence, $\frac{d^n}{dz^n}\left(\mathcal{J}\right)$ is a (closed) subspace of $H^p$ if and only if $\mathcal J$ is a (closed) subspace of $S_{n,0}^p$\,.

If $\frac{d^n}{dz^n}\left(\mathcal{J}\right)$ is an invariant subspace of $T$ on $H^p$, according to Lemma~\ref{le5},
$$M_z(\mathcal J)=\left(\frac{d^n}{dz^n}\right)^{-1}T\frac{d^n}{dz^n}\left(\mathcal J\right)\subset\left(\frac{d^n}{dz^n}\right)^{-1}\frac{d^n}{dz^n}\left(\mathcal J\right)=\mathcal J\,,$$
which implies that $\mathcal J$ is an invariant subspace of $M_z$ on $S_{n,0}^p$\,.

Conversely, if $\mathcal J$ is an invariant subspace of $M_z$ on $S_{n,0}^p$, then
$$T\left(\frac{d^n}{dz^n}\left(\mathcal{J}\right)\right)=\frac{d^n}{dz^n}M_z\left(\mathcal{J}\right)\subset\frac{d^n}{dz^n}\left(\mathcal{J}\right)\,,$$
which implies that $\frac{d^n}{dz^n}\left(\mathcal{J}\right)$ is an invariant subspace of $T$ on $H^p$.

Thus, we complete the proof.
\end{proof}

\end{document}